\documentclass[8pt]{article}
\usepackage{indentfirst,latexsym,bm}
\usepackage{amsfonts}
\usepackage{amssymb}
\usepackage{times}
\usepackage[leqno]{amsmath}
\usepackage{dsfont}
\usepackage[all]{xy}
\usepackage{color,xcolor}
\usepackage{amsthm}
\usepackage{hyperref}
\usepackage{tikz-cd}
\usepackage{tikz}
\usetikzlibrary{matrix,arrows.meta, positioning}

\usepackage[text={14cm,20cm}, centering]{geometry}

\begin{document}
\newcommand{\Z}{\mathbb{Z}}
\newcommand{\rank}{\text{rank}}
\date{}
\newtheorem{theo}{Theorem}[section]
\newtheorem{prop}[theo]{Proposition}
\newtheorem{lemm}[theo]{Lemma}
\newtheorem{coro}[theo]{Corollary}
\theoremstyle{definition}
\newtheorem{defi}[theo]{Definition}
\newtheorem{exam}[theo]{Example}
\newtheorem{rema}[theo]{Remark}

\title{Irreducible $\mathbb{Z}_+$-modules over some $\mathbb{Z}_+$-rings}
\author{Yue Meng\thanks{Email:\,yuemengmath@163.com}\\{School  of Mathematical Science,  Yangzhou University}}
\maketitle

\abstract
Let $d_1,\ldots,d_r$ be pairwise relatively prime positive square-free integers, with  $d_j\geq2$ for all $1\leq j\leq r$. Using elementary matrices and combinatorial mathematics, we give a complete classification of the  irreducible  $\mathbb{Z}_+$-modules over the domain $\mathbb{Z}[d^{\frac{1}{N_1}}_1,d^{\frac{1}{N_2}}_2,\dots,d^{\frac{1}{N_r}}_r]$, where $N_i \geq2$ for all $ 1\leq i \leq r$. Furthermore, we explicitly construct all of these irreducible $\mathbb{Z}_+$-modules.

\noindent {\bf Keywords:} Irreducible $\mathbb{Z}_+$-module; $\mathbb{Z}_+$-ring; non-negative integer matrix
\section{Introduction}
Tensor categories are regarded as the categorical counterparts of  rings in category theory, forming a higher-dimensional categorification of groups and rings. The combinatorics needed to study tensor categories is the theory of $\mathbb{Z}_+$-rings, i.e., rings equipped with a basis where the structure constants are non-negative integers. Such rings serve as Grothendieck rings for tensor categories, and it has been observed that numerous properties of tensor categories are fundamentally combinatorial in nature, originating  from specific properties of the Grothendieck ring. The concept of $\mathbb{Z}_+$-ring and $\mathbb{Z}_+$-module was introduced by G. Lusztig in \cite{Lu}, and these rings have been extensively studied in \cite{Davy,EK,O}. By definition, a $\mathbb{Z}_+$-module over a given $\mathbb{Z}_+$-ring corresponds to a set of matrices with non-negative integer coefficients that satisfy certain conditions.

Let $G$ be a finite group. The irreducible $\mathbb{Z}_+$-modules over the group ring $\mathbb{Z}[G]$ are characterized in \cite{EK}. Furthermore, V. Ostrik has demonstrated that there are finitely many inequivalent classes of irreducible $\mathbb{Z}_+$-modules for a given $\mathbb{Z}_+$-ring of finite rank \cite{O}, which makes it feasible to study the irreducible $\mathbb{Z}_+$-modules for some $\mathbb{Z}_+$-rings of finite rank. However, in general, classifying the irreducible $\mathbb{Z}_+$-modules of an arbitrary $\mathbb{Z}_+$-ring is a non-trivial endeavor, see \cite{YL}\cite{CLZ}\cite{CCML} for example.

In  \cite{CLZ}, they use the  matrix  method to study the irreducible $\Z_+$-modules of the domain $R=\Z[d^{\frac{1}{2}}_1,d^{\frac{1}{2}}_2,\dots,d^{\frac{1}{2}}_r]$, where $d_1,\ldots,d_r$ are positive square-free integers that are pairwise relatively prime. 
In this note, we generalized the results of \cite{CLZ} by employing more combinatorial methods, which were not utilized in \cite{CLZ}.
Applying them to study the irreducible $\Z_+$-modules of the domain $R=\Z[d^{\frac{1}{N_1}}_1,d^{\frac{1}{N_2}}_2,\dots,d^{\frac{1}{N_r}}_r]$, where $N_i\geq2$ for all $1\leq i \leq r$, and $d_1,\ldots,d_r$ are positive square-free integers that are pairwise relatively prime. Then $1$ and $d^{\frac{i_1}{N_{j_1}}}_{j_1}d^{\frac{i_2}{N_{j_2}}}_{j_2}\cdots d^{\frac{i_{s}}{N_{j_{s}}}}_{j_{s}}$, where $0\leq i_s\leq N_{j_{s}}-1$ for all $1\leq s\leq r$, form a  $\Z_+$-basis of $R$, so $\rank(R)=N_1N_2\cdots N_r$. It is worthwhile to note that  the condition that $d_1,\ldots,d_r$ are pairwise relatively prime is necessary. Indeed, if $d_i$ and $d_j$ have  a common divisor that is greater than $1$, then  it is not easy to give a complete form of the $\Z_+$-basis of $R$, and some of the conclusions of this note also  fail, see Remark \ref{Example}, for example.

Let $d$ be a positive square-free integer. We show that any irreducible $\Z_+$-module of $\Z[d^{\frac{1}{N}}]$ is isomorphic to  a free $\Z$-module
\begin{align*}
&\mathbb{Z} \langle \prod\limits_{i=1}^{N-1}c_{N-i}^{\frac{i}{N}},\frac{1}{c_{1}}d^{\frac{1}{N}}\prod\limits_{i=1}^{N-1}c_{N-i}^{\frac{i}{N}},\frac{1}{c_1c_2}d^{\frac{2}{N}}\prod\limits_{i=1}^{N-1}c_{N-i}^{\frac{i}{N}} ,\dots,\frac{1}{c_1c_2\cdots c_{N-1}}d^{\frac{N-1}{N}}\prod\limits_{i=1}^{N-1}c_{N-i}^{\frac{i}{N}} \rangle 
\\=& \Z\langle 
c^{\frac{1}{N}}_{N-1}\cdots c^{\frac{N-2}{N}}_{2} c^{\frac{N-1}{N}}_{1},c^{\frac{1}{N}}_{N}\cdots c^{\frac{N-2}{N}}_{3} c^{\frac{N-1}{N}}_{2},c^{\frac{1}{N}}_{1}\cdots c^{\frac{N-2}{N}}_{4} c^{\frac{N-1}{N}}_{3},\dots,c^{\frac{1}{N}}_{N-2}\cdots c^{\frac{N-2}{N}}_{1} c^{\frac{N-1}{N}}_{N}
\rangle,
\end{align*}
for some positive divisors $c_1,c_2,\dots,c_N$ of $d$ and $c_1c_2\cdots c_N=d$ (see Corollary \ref{resultonZ[d]}). Given an irreducible $\Z_+$-module $M$ over $R=\Z[d^{\frac{1}{N_1}}_1,d^{\frac{1}{N_2}}_2,\dots,d^{\frac{1}{N_r}}_r]$, for any $1\leq j\leq r$, we prove that $M$ is isomorphic to a direct sum  of
\begin{align*}
\mathbb{Z}\langle 
c^{\frac{1}{N_j}}_{(N_j-1)j}\cdots c^{\frac{N_j-2}{N_j}}_{2j} c^{\frac{N_j-1}{N_j}}_{1j},c^{\frac{1}{N_j}}_{N_jj}\cdots c^{\frac{N_j-2}{N_j}}_{3j} c^{\frac{N_j-1}{N_j}}_{2j},\dots,c^{\frac{1}{N_j}}_{(N_j-2)j}\cdots c^{\frac{N_j-2}{N_j}}_{1j} c^{\frac{N_j-1}{N_j}}_{N_jj}
\rangle,
\end{align*}
as a $\Z[d_j^{\frac{1}{N_j}}]$-module where $c_{1j}c_{2j}\cdots c_{N_{j}j}=d_j$ and  that $M$ has the same rank as $R$ in Theorem \ref{TypeRank}, then we give a complete classification of irreducible $\Z_+$-modules of $R$ in Theorem \ref{maintheorem}.

The paper is organized as follows. In Section \ref{Preliminary}, we review the definitions of $\Z_+$-ring and $\Z_+$-module. In Section \ref{section3}, we first give all the non-negative integer solutions of the matrix equation $A^N=dI_n$, where $d$ is a square free integer and $N\geq 2$ (Theorem \ref{oned}). We then classify the irreducible $\Z_+$-modules over $\Z[d^{\frac{1}{N}}]$ (Corrollary \ref{resultonZ[d]}). Finally, we classify the irreducible $\Z_+$-modules over $\Z[d^{\frac{1}{N_1}}_1,d^{\frac{1}{N_2}}_2,\dots,d^{\frac{1}{N_r}}_r]$, where $d_1,\dots,d_r$ are pairwise coprime square-free integers and $N_i\geq 2$ for $1\leq i\leq r$ (Theorem \ref {TypeRank} and Theorem \ref {maintheorem}).

\section{Preliminaries}\label{Preliminary}

Let $\Z_+$ denote the semi-ring of non-negative integers, and $M_n(\Z_+)$ the set of square matrices of order $n$ with coefficients being non-negative integers, which is a semiring. In this section, we recall the definitions of $\Z_+$-ring and $\Z_+$-module, see \cite{EGNO,O}.
\begin{defi}Let $R$ be a unital domain such that $R$ is free as a $\Z$-module with basis $\{x_1,\dots, x_n\}$, where $x_1=1$. If there exist $c_{ij}^k\in\Z_+$ such that  $x_ix_j=\sum_{k=1}^nc_{ij}^kx_k$ for all $1\leq i,j\leq n$, then $R$ is called a unital $\Z_+$-ring, $c_{ij}^k$ are structure constants of $R$, and $n$ is the rank of $R$.
\end{defi}
\begin{defi}Let $R$ be a unital $\Z_+$-ring with basis $\{x_1,\dots, x_n\}$.  An $R$-module $M$ is called a  $\Z_+$-module of $R$ with $\Z_+$-basis $\{m_1,m_2,\dots,m_t\}$, if $x_im_j=\sum_{k=1}^ta_{ij}^km_k$, $a_{ij}^k\in\Z_+$, for $1\leq i\leq n$, $1\leq j\leq t$, and $t$ is called the rank of   $M$.

If a $\Z$-submodule  $M_1$ of $M$ spanned by $\{m_{j_1},\dots,m_{j_s}\}\subseteq\{m_1,m_2,\dots,m_t\}$  is an
$R$-submodule of $M$, then we say that $M_1$ is a $\Z_+$-submodule of $M$. $M_1$ is also called the $\Z_+$-submodule generated by $\{m_{j_1},\dots,m_{j_s}\}$, and it will be denoted by $(m_{j_1},\dots,m_{j_s})$ below. A $\Z_+$-module $M$ is irreducible if it has no proper $\Z_+$-submodule.
\end{defi}
It is easy to see that  $M$ is a $\Z_+$-module of ring $R$ is equivalent to that there exists a homomorphism of semiring from $R$ to $M_n(\Z_+)$, given by $x_{i} \mapsto (a_{ij}^k)_{n \times n}$ for $1\leq i\leq n$, where $n$ is the rank of $M$.

\begin{defi}
Let $M,M'$ be $\Z_+$-modules of $\Z_+$-ring $R$. Assume that  $\{m_1,\dots,m_t\}$ and  $\{m_1',\dots,m_t'\}$ are the set of $\Z_+$-bases of $M$ and $M'$, respectively; moreover, assume $x_im_j=\sum_{k=1}^ta_{ij}^km_k$ and $x_im_j'=\sum_{l=1}^tb_{ij}^lm_l'$. If there exists an $R$-module homomorphism $\psi:M\to M'$ and $\sigma\in S_t$ such that $\psi(m_i)=m_{\sigma(i)}'$, where $S_t$ is the symmetric group, then  $M$ is said to  be isomorphic to $M'$, and denoted by $M\cong M'$.
\end{defi}

\begin{rema}\label{transfer}
Assume $\psi:M\to M^{\prime}$ is an isomorphism of $\Z_+$-modules of $R$. Then $\psi(x_im_j)=x_i\psi(m_j) =x_im_{\sigma(j)}'$, that is $\sum_{k=1}^ta_{ij}^km_{\sigma(k)}'=\sum_{l=1}^tb_{i\sigma(j)}^lm_l'$.
 Let  $A_i:=(a_{ij}^k)_{t\times t}$ and $B_i:=(b_{ij}^k)_{t\times t}$, where $1\leq i\leq n$, $1\leq j,k\leq t$. Therefore, two $\Z_+$-modules $M$ and  $M'$ are isomorphic if and only if there exists $\sigma\in S_t$  such that  $a_{ij}^k=b_{i\sigma(j)}^{\sigma^{-1}(k)}$. Equivalently, there exists a permutation matrix $P$ such that $B_i=PA_iP^{-1}$ for all $1\leq i\leq n$.  Thus, in this note, we could use the matrix method to classify the isomorphism classes of the  irreducible $\Z_+$-modules over the domain $\Z[d^{\frac{1}{N_1}}_1,d^{\frac{1}{N_2}}_2,\dots,d^{\frac{1}{N_r}}_r]$, where $N_i\geq2$ for all $1\leq i \leq r$, and $d_1,\ldots,d_r$ are pairwise relatively prime  positive square-free integers with $d_j\geq 2$ for all $1\leq j\leq r$.
\end{rema}

\section{The main theorem}\label{section3}
In this section, we give a complete classification of irreducible $\Z_+$-modules over the domain $\Z[d^{\frac{1}{N_1}}_1,d^{\frac{1}{N_2}}_2,\dots,d^{\frac{1}{N_r}}_r]$, where $d_1,\dots,d_r$ are pairwise coprime square-free integers and $N_i\geq 2$ for $1\leq i\leq r$.
For $1\leq i\neq j\leq n,$ let $E(i,j)$ be the  elementary matrix obtained by interchanging the $i$-th row and $j$-th row of the identity matrix $I_n$. 

Consider dividing the set of positive integers
$\{1,2,3,\dots,Nt\}$
 into $t$ pairwise disjoint subsets, where $N\geq 2$ and $t\geq 1$. These subsets are denoted by:
$$
 \{ 
i_{1,1},i_{2,1},\dots ,i_{N,1}\},\{ 
i_{1,2},i_{2,2},\dots ,i_{N,2}\},\dots,\{ i_{1,t},i_{2,t},\dots, i_{N,t}\}.
\nonumber
$$
The corresponding $N\text-$cycles are
$
 (i_{1,1}i_{2,1}\cdots i_{N,1}),(i_{1,2}i_{2,2}\cdots i_{N,2}),\dots, (i_{1,t}i_{2,t}\cdots i_{N,t}). 
\nonumber
$
In the following text, we omit commas for brevity and write them as: $$
 (i_{11}i_{21}\cdots i_{N1}),(i_{12}i_{22}\cdots i_{N2}),\dots, (i_{1t}i_{2t}\cdots i_{Nt}). 
\nonumber
$$

To establish the first theorem, we require the following three lemmas.

\begin{lemm}\label{simpleLemm1}  
For any $N \geq 2$ and $t\geq 1$ there exist $2\text-$cycles denoted by
$(i_1 j_1),\dots,(i_s j_s)$, such that
\begin{small}
\begin{equation}
\begin{aligned}
&[(i_s j_s)\cdots(i_1 j_1)][\prod\limits_{l=1}^{t}(i_{1l}i_{2l}\cdots i_{Nl})][(i_s j_s) \cdots(i_1 j_1)]^{-1}
\\
&=\prod\limits_{k=1}^{t}((N(k-1)+1)(N(k-1)+2)\cdots(Nk)).
\nonumber
\end{aligned}
\end{equation}
\end{small}
\begin{proof}
If $1$ is contains in $(i_{1l}i_{2l}\cdots i_{Nl})$ for some $1\leq l \leq t$, then we can rewrite the cycle as $(1i_{2l}\cdots i_{Nl})$. If $i_{2l} =2$,  no action is required; if $i_{jl}=2$ for some $3\leq j \leq N$, then
\begin{align*}
(2i_{2l})[\cdots(1i_{2l}\cdots i_{Nl})\cdots](2i_{2l})^{-1}=[\cdots(12\cdots )\cdots].
\end{align*}
If $i_{jl}\neq 2$ for all $2\leq j \leq N$, and $2$ is contained in $(i_{1m}i_{2m}\cdots i_{Nm})$ for $m\neq l$, we can rewrite the cycle as $(2i_{2m}\cdots i_{Nm})$, then 
\begin{align*}
(2i_{2l})[\cdots(1i_{2l}\cdots i_{Nl})\cdots(2i_{2m}\cdots i_{Nm})\cdots](2i_{2l})^{-1}=[\cdots(12i_{3l}\cdots i_{Nl})\cdots].
\end{align*}
By repeating this process for the set $\{3,4,5,\dots,Nt\}$, we can prove the lemma.
\end{proof}

\end{lemm}

\begin{lemm}\label{simpleLemm2}  Let $ d\geq 2$ be an arbitrary positive square-free integer, and let $A =(a_{ij}) \in M_{n}(\mathbb{Z}_{+})$, where $n\geq 2$. If $A^{N}=dI_{n}$ for some integer $N\geq 2$, then all the diagonal elements of $A$ are zero.
\begin{proof}
Assuming \( a_{11} \neq 0 \), we analyze the first column of \( A^N \). Each \((i,1)\)-entry of $A^N$ contains a term of the form \( a_{i1} \underbrace{a_{11} \cdots a_{11}}_{N-1} \). Since all these terms are zero, it follows that \( a_{i1} = 0 \) for all \( 2 \leq i \leq n \). Similarly, by examining the first row of \( A^N \), we know that each \((1,j)\)-entry of $A^N$ contains a term of the form \( \underbrace{a_{11} \cdots a_{11}}_{N-1} a_{1j} \). These terms are also zero, leading to the conclusion that \( a_{1j} = 0 \) for all \( 2 \leq j \leq n \). Consequently, the \((1,1)\)-entry of \( A^N \) is \( a_{11}^N \). Since \( a_{11}^N \neq d \) (given that \( d \) is square-free), it is a contradiction. Therefore, \( a_{11} = 0 \). By the same reasoning, we can prove that \( a_{ii} = 0 \) for all \( 1 \leq i \leq n \).
\end{proof}

\end{lemm}

\begin{lemm}\label{simpleLemm3}  Let $ d\geq 2$ be an arbitrary positive square-free integer, and let $A =(a_{ij}) \in M_{n}(\mathbb{Z}_{+})$, where $n\geq 2$. If $A^{N}=dI_{n}$ for some integer $N\geq 2$, then there is exactly one non-zero element in each row and column of $A$.
\begin{proof}
If $n=2$, it is exactly Lemma \ref{simpleLemm2}. Assume $n\geq 3$ below.

Suppose the first row of $A$ has two nonzero entries, denoted as $a_{1p},a_{1q}$. It follows from Lemma \ref{simpleLemm2} that $p\neq 1$ and $q\neq 1$. We claim that there exists a summand of the $(i,1)$-entry, of $A^{N-1}$ that has the form $$a_{ii_1}a_{i_1i_2}a_{i_2i_3}\cdots a_{i_{N-3}i_{N-2}}\\a_{i_{N-2} 1}\neq 0,$$ for some $1\leq i\leq n$. If not, all terms of the form $a_{ii_1}a_{i_1i_2}a_{i_2i_3}\cdots a_{i_{N-3}i_{N-2}}a_{i_{N-2} 1}=0$, then the first column of $A^{N-1}$ is zero, which is impossible. Therefore, the $(i,p)$-entry and the $(i,q)$-entry contain non-zero summands of the form $a_{ii_1}a_{i_1i_2}a_{i_2i_3}\cdots a_{i_{N-3}i_{N-2}}a_{i_{N-2} 1}a_{1p} $ and $a_{ii_1}a_{i_1i_2}a_{i_2i_3}\cdots a_{i_{N-3}i_{N-2}}a_{i_{N-2} 1}a_{1q}$.
If $i\neq p$ and $i\neq q$, the $(i,p),(i,q)$-entry of $A^{N}$ are not zero; if $i=p$, the $(p,q)$-entry of $A^{N}$ is not zero; if $i=q$, the $(q,p)$-entry of $A^{N}$ is not zero, this is a contradiction. Thus the first row can only have one non-zero element. Using the same method, we can prove that any row or column of 
$A$ contains only one non-zero element.
\end{proof}

\end{lemm}

\begin{theo}\label{oned}Let $ d=p_{1}p_{2}\cdot \cdot \cdot p_{m}\geq 2$ be an arbitrary positive square-free integer, where $p_{j}$ are distinct primes. Let $A$ be a non-negative integer matrix such that $A^{N}=dI_{n}$ for some integer $N\geq 2$. Then $n=tN$, where $t \geq 1$ and there exists a permutation matrix $P=E(i_1,j_1)E(i_2,j_2)\cdot \cdot \cdot E(i_s,j_s)$ such that 
$$PAP^{-1}=\begin{bmatrix}
A_{1} &  &  & \\
 & A_{2} &  &\\
 &  & \ddots &\\
 &  &  & A_{t}
\end{bmatrix},$$
where each block $A_k$ is given by
$$A_{k}=
\begin{bmatrix}
0 & \cdots & 0 & c_{Nk}\\
c_{1k} & \cdots & 0 & 0\\
\vdots & \ddots & \vdots &\vdots\\
0 & \cdots &c_{(N-1)k} & 0
\end{bmatrix}
,$$ for $1\leq k \leq t.$ The entries $c_{ik}$ are positive integers satisfying $c_{1k}c_{2k}\cdot \cdot \cdot c_{Nk}=d$, for $1\leq i\leq N$.
\end{theo}\begin{proof}
We construct $c_{1l}E_{ i_{2l}i_{1l}}+c_{2l}E_{i_{3l} i_{2l}}+\cdots+c_{Nl}E_{i_{1l} i_{Nl}}$ for each cycle $(i_{1l}i_{2l}\cdots i_{Nl})$, where $c_{1l}c_{2l}\cdots c_{Nl}=d$ for $1 \leq l\leq t$. It is evident that $(c_{1l}E_{ i_{2l}i_{1l}}+c_{2l}E_{i_{3l} i_{2l}}+\cdots+c_{Nl}E_{i_{1l} i_{Nl}})^N=c_{1l}c_{2l}\cdots c_{Nl}(E_{i_{1l}i_{1l}}+E_{i_{2l}i_{2l}}+\cdots+E_{i_{Nl}i_{Nl}})$. Let 
\begin{align*}
X=\sum\limits_{l=1}^t (c_{1l}E_{ i_{2l}i_{1l}}+c_{2l}E_{i_{3l} i_{2l}}+\cdots+c_{Nl}E_{i_{1l} i_{Nl}}),
\end{align*} 
it is clear that  $X^N=d(E_{11}+E_{22}+\cdots+E_{NtNt})=dI_{Nt}$. 

Next, we show that $X=\sum_{l=1}^t (c_{1l}E_{ i_{2l}i_{1l}}+c_{2l}E_{i_{3l} i_{2l}}+\cdots+c_{Nl}E_{i_{1l} i_{Nl}})$ constitutes all the solutions of equation $A^N=dI_n$.  
 From Lemma \ref{simpleLemm3}, we can write $$X=x_{\square1}E_{\square1}+x_{\square2}E_{\square2}+\cdot\cdot\cdot+x_{\square n}E_{\square n},$$  each $\square$ is a number from the sets $\{1,2,3,\dots,n\}$, they are all different, and it follows from Lemma \ref{simpleLemm2} that every number in a $\square$ cannot be the same as the number immediately following it. $X^N$ are sums of those non-zero terms $x_{\square i_1}x_{\square i_2}\cdots x_{\square i_N}E_{\square i_1}E_{\square i_2}\cdots E_{\square i_N}$, where $1\leq i_j \leq n$ for all $1\leq j \leq N$. All indices \( i_j \) must be distinct, which follows from the fact that if \( E_{\square i_m} = E_{\square i_n} \) for any \( 1 \leq m \neq n \leq N \), then \( x_{\square i_m} x_{\square i_n} = 0 \) by Lemma \ref{simpleLemm3}. Consequently, the only non-zero terms are of the form \( x_{i_N i_1} x_{i_1 i_2} \cdots x_{i_{N-1} i_N} E_{i_N i_1} E_{i_1 i_2} \cdots E_{i_{N-1} i_N} \), which implies that $x_{i_N i_1} \neq 0 $, $ x_{i_1 i_2} \neq 0,$ $\dots,$ \( x_{i_{N-1} i_N} \neq 0 \) and \( x_{i_N i_1} x_{i_1 i_2} \cdots x_{i_{N-1} i_N} = d \). Therefore, $X$ is the sum of the form 
$$x_{i_N i_1}E_{i_N i_1}+x_{i_1 i_2}E_{i_1 i_2}+\cdots+x_{i_{N-1} i_N}E_{i_{N-1} i_N},$$ 
for the cycle $(i_Ni_{N-1}i_{N-2}\cdots i_2i_1)$, where $x_{i_N i_1} x_{i_1 i_2} \cdots x_{i_{N-1} i_N} = d$. So $n=tN$ for some integer $t$.

We establish a mapping 
$
\phi: S_{Nt}\rightarrow M_{Nt}(\mathbb{Z}_{+}),
\nonumber
$
restrict the map to $\prod_{l=1}^{t}(i_{1l}i_{2l}\cdots i_{Nl})$, it is given by
$$
\phi(\prod\limits_{l=1}^{t}(i_{1l}i_{2l}\cdots i_{Nl}))=\sum\limits_{l=1}^t \phi(i_{1l}i_{2l}\cdots i_{Nl}),
$$
where $$\phi(i_{1l}i_{2l}\cdots i_{Nl})=c_{1l}E_{ i_{2l}i_{1l}}+c_{2l}E_{i_{3l} i_{2l}}+\cdots+c_{Nl}E_{i_{1l} i_{Nl}},$$
and $c_{1l}c_{2l}\cdots c_{Nl}=d$ for $1\leq l\leq t$.

From the previous paragraph, we can write $$X=\sum\limits_{l=1}^t \phi(i_{1l}i_{2l}\cdots i_{Nl})=\phi(\prod\limits_{l=1}^{t}(i_{1l}i_{2l}\cdots i_{Nl}))$$.

Next, we will use the same method as in Lemma \ref{simpleLemm1} for $\sum_{l=1}^t \phi(i_{1l}i_{2l}\cdots i_{Nl})
$.
\par
If $1$ is contains in $(i_{1l}i_{2l}\cdots i_{Nl})$ for some $1\leq l \leq t$, then we can rewrite the cycle as $(1i_{2l}\cdots i_{Nl})$.
If $i_{2l} =2$,  no action is taken; if $i_{jl}=2$ for some $3\leq j \leq N$, then $$E(2,i_{2l})[\cdots +\phi(1i_{2l}\cdots i_{Nl})+\cdots]E(2,i_{2l})^{-1}=[\cdots\phi(12\cdots )\cdots].$$ If $i_{jl}\neq 2$ for any $2\leq j \leq N$, and $2$ is contained in $(i_{1m}i_{2m}\cdots i_{Nm})$ for $m\neq l$,  we rewrite the cycle as $(2i_{2m}\cdots i_{Nm})$, then $$E(2,i_{2l})[\cdots +\phi(1i_{2l}\cdots i_{Nl})+\phi(2i_{2m}\cdots i_{Nm})+\cdots]E(2,i_{2l})^{-1}=[\cdots\phi(12i_{3l}\cdots i_{Nl})\cdots].$$

By continuing this process for the set $\{3,4,5,\dots,Nt\}$, we obtain the following equations
\begin{align*}
&[E(i_s,j_s)\cdots E(i_1,j_1)][\sum\limits_{l=1}^t \phi(i_{1l}i_{2l}\cdots i_{Nl})][E(i_s,j_s)\cdots E(i_1,j_1)]^{-1}\\
&=\sum\limits_{k=1}^t \phi((N(k-1)+1)(N(k-1)+2)\cdots(Nk))\\
&=\phi(\prod\limits_{k=1}^{t}((N(k-1)+1)(N(k-1)+2)\cdots(Nk)))\\
&=\phi([(i_s j_s)\cdots(i_1 j_1)][\prod\limits_{l=1}^{t}(i_{1l}i_{2l}\cdots i_{Nl})][(i_s j_s) \cdots(i_1 j_1)]^{-1}).
\end{align*}
These equations indicate that the procedure outlined here is entirely analogous to that described in Lemma \ref{simpleLemm1}.

Reorder the coefficients $c_{ij}$, we eventually obtain the form
$$
\sum\limits_{k=1}^t \phi((N(k-1)+1)(N(k-1)+2)\cdots(Nk))=\begin{bmatrix}
X_{1} &  &  & \\
 & X_{2} &  &\\
 &  & \ddots &\\
 &  &  & X_{t}
\end{bmatrix},
$$
where
$X_k=
\begin{bmatrix}
0 & \cdots & 0 & c_{Nk}\\
c_{1k} & \cdots & 0 & 0\\
\vdots & \ddots & \vdots &\vdots\\
0 & \cdots &c_{(N-1)k} & 0
\end{bmatrix}$ for $1\leq k\leq t$.
Thus, there exists a permutation matrix
 $P=E(i_1,j_1)\cdots E(i_s,j_s)$ such that 
$$PXP^{-1}=P(\sum\limits_{l=1}^t \phi(i_{1l}i_{2l}\cdots i_{Nl}))P^{-1}=
\begin{bmatrix}
X_{1} &  &  & \\
 & X_{2} &  &\\
 &  & \ddots &\\
 &  &  & X_{t}
\end{bmatrix}.$$

This completes the proof of the theorem.
\end{proof}

\begin{rema}\label{remark1}
If $d=1$, the theorem above is also correct and reduces to Lemma \ref{simpleLemm1}, by letting
$$\phi(i_{1l}i_{2l}\cdots i_{Nl})=E_{ i_{2l}i_{1l}}+E_{i_{3l} i_{2l}}+\cdots+E_{i_{1l} i_{Nl}},$$
for all $1\leq l\leq t$.

\end{rema}

Let \( d = p_1 p_2 \cdots p_m \geq 2 \) be a positive square-free integer. We construct an irreducible 
\(\mathbb{Z}_+\)-module of \(\mathbb{Z}[d^{\frac{1}{N}}]\) as follows
\begin{align*}
&\mathbb{Z} \langle \prod\limits_{i=1}^{N-1}c_{N-i}^{\frac{i}{N}},\frac{1}{c_{1}}d^{\frac{1}{N}}\prod\limits_{i=1}^{N-1}c_{N-i}^{\frac{i}{N}},\frac{1}{c_1c_2}d^{\frac{2}{N}}\prod\limits_{i=1}^{N-1}c_{N-i}^{\frac{i}{N}} ,\dots,\frac{1}{c_1c_2\cdots c_{N-1}}d^{\frac{N-1}{N}}\prod\limits_{i=1}^{N-1}c_{N-i}^{\frac{i}{N}} \rangle 
\\= &\Z\langle 
c^{\frac{1}{N}}_{N-1}\cdots c^{\frac{N-2}{N}}_{2} c^{\frac{N-1}{N}}_{1},c^{\frac{1}{N}}_{N}\cdots c^{\frac{N-2}{N}}_{3} c^{\frac{N-1}{N}}_{2},c^{\frac{1}{N}}_{1}\cdots c^{\frac{N-2}{N}}_{4} c^{\frac{N-1}{N}}_{3},\dots,c^{\frac{1}{N}}_{N-2}\cdots c^{\frac{N-2}{N}}_{1} c^{\frac{N-1}{N}}_{N}
\rangle,
\end{align*}
where \( c_i \) are positive divisors of \( d \) for \( 1 \leq i \leq N \) and satisfy \( c_1 c_2 \cdots c_N = d \). 
\begin{lemm}\label{simpleLemm4} Let $\Pi \subset S_N$ denote the subgroup of permutations $\sigma \in S_N$ fixing $N$ (i.e.,$\sigma(N) = N$). For each $\sigma \in \Pi$, we apply $\sigma$ to the lower indices of the generators in the above $\mathbb{Z}_+$-module. Let $j$ denote the number of trivial divisors $c_i=1$ in the $N$-tuple 
$(c_1,\dots,c_{N-1},c_N)$. Then, under these permutation actions, the images generate exactly $\frac{(N-1)!}{j!}$ non-isomorphic irreducible 
$\mathbb{Z}_+$-modules.
\begin{proof} 
It is easy to see that two $\mathbb{Z}_{+}\text-$modules of the above form are isomorphic if and only if they are equal as free $\mathbb{Z}\text-$modules. 
 
Without loss of generality, we assume \( c_N \neq 1 \). It is easy to see that the first term of the basis determines the entire basis. The number of distinct images of \( c^{\frac{1}{N}}_{\sigma(N-1)} \cdots c^{\frac{N-2}{N}}_{\sigma(2)} c^{\frac{N-1}{N}}_{\sigma(1)} \) corresponds to the number of permutations of \( c_1, \dots, c_{N-1} \). Given that there are \( j \) elements equal to $1$ among \( c_1, \dots, c_{N-1} \), where \( 1 \leq j \leq N-1 \),  the number of distinct images of \( c^{\frac{1}{N}}_{\sigma(N-1)} \cdots c^{\frac{N-2}{N}}_{\sigma(2)} c^{\frac{N-1}{N}}_{\sigma(1)} \) under all permutations \( \sigma \) is \( \frac{(N-1)!}{j!} \), where \( \sigma \) ranges over all elements of the group \( \Pi \). Consequently, there are \( \frac{(N-1)!}{j!} \) non-isomorphic irreducible \(\mathbb{Z}_+\)-modules. 
\end{proof}

\end{lemm}

\begin{rema}\label{remark1}
Since two $\mathbb{Z}_{+}\text-$modules of the above form are isomorphic if and only if they are equal as free $\mathbb{Z}\text-$modules, it is convenient to order the bases by setting $c_N=\max\{c_1,\dots,c_N\}$. \end{rema}

\begin{exam}\label{ex1}
 Let $N=4$, $d=6$ and $(c_1,c_2,c_3,c_4)=(1,1,2,3)$. Then $\Pi=\{(1),(123),(132)\}$. Denote $\sigma=(123)$ and $\tau=(132)$. By the above lemma there are $\frac{(4-1)!}{2!}=3$ non-isomorphic irreducible \(\mathbb{Z}_+\)-modules, and they are
\begin{align*}
(c_1,c_2,c_3,c_4)=(1,1,2,3) \leftrightarrow \mathbb{Z}\langle 
2^{\frac{1}{4}},3^{\frac{1}{4}}2^{\frac{2}{4}},3^{\frac{2}{4}}2^{\frac{3}{4}},3^{\frac{3}{4}}
\rangle,\\
(c_{\sigma(1)},c_{\sigma(2)},c_{\sigma(3)},c_{\sigma(4)})=(1,2,1,3)\leftrightarrow \mathbb{Z}\langle 
2^{\frac{2}{4}},3^{\frac{1}{4}}2^{\frac{3}{4}},3^{\frac{2}{4}},2^{\frac{1}{4}}3^{\frac{3}{4}}
\rangle,\\
(c_{\tau(1)},c_{\tau(2)},c_{\tau(3)},c_{\tau(4)})=(2,1,1,3)\leftrightarrow \mathbb{Z}\langle 
2^{\frac{3}{4}},3^{\frac{1}{4}},2^{\frac{1}{4}}3^{\frac{2}{4}},2^{\frac{2}{4}}3^{\frac{3}{4}}
\rangle.
\end{align*}
\end{exam}

The following lemma is a standard result in combinatorial mathematics, and we refer the reader to \cite{CPJ}.

\begin{lemm}\label{simpleLemm5}Let $(N)_k:=N(N-1)\cdots(N-k+1)$, $S(m,k):=\frac{1}{k!}\sum_{j=0}^k(-1)^{k-j}\tbinom{k}{j}j^m$, where $N,k,m$ are positive integers and $m\geq k.$ Then  $N^m=\sum_{k=1}^mS(m,k)(N)_k$.  $S(m,k)$ is the Stirling number of the second kind, which is the number of partitions of  $\{1,2,3,\dots,m\}$ with $k$ non-empty parts. 
\end{lemm}

\begin{coro}\label{resultonZ[d]}Let \( d = p_1 p_2 \cdots p_m \geq 2 \) be a positive square-free integer. Then, any irreducible \(\mathbb{Z}_+\)-module of \(\mathbb{Z}[d^{\frac{1}{N}}]\) is isomorphic to some
\[
\mathbb{Z} \langle c^{\frac{1}{N}}_{N-1} \cdots c^{\frac{N-2}{N}}_{2} c^{\frac{N-1}{N}}_{1}, c^{\frac{1}{N}}_{N} \cdots c^{\frac{N-2}{N}}_{3} c^{\frac{N-1}{N}}_{2}, \dots, c^{\frac{1}{N}}_{N-2} \cdots c^{\frac{N-2}{N}}_{1} c^{\frac{N-1}{N}}_{N} \rangle,
\]
where \( c_i \) are positive divisors of \( d \) for \( 1 \leq i \leq N \), and satisfy \( c_1 c_2 \cdots c_N = d \). Consequently, \(\mathbb{Z}[d^{\frac{1}{N}}]\) has exactly \( N^{m-1} \) irreducible \(\mathbb{Z}_+\)-modules up to isomorphism.
\end{coro}
 \begin{proof}Let $M$ be an arbitrary irreducible $\mathbb{Z}_{+}\text-$module of 
$\mathbb{Z}[d^{\frac{1}{N}}]$ with $\mathbb{Z}_{+}\text-$basis $\{m_{1},\dots, m_{n}\}$. Let $A$ be the non-negative matrix determined by $d^{\frac{1}{N}}$. Then it follows from Theorem \ref{oned} that there exists a permutation matrix $P$ such that $PAP^{-1}$ is a block diagonal matrix. Then $M$ can be decomposed as a direct sum of $\mathbb{Z}_{+}\text-$submodules of rank $N$ by Remark \ref{transfer}. However, $M$ is irreducible, so rank$(M)=N$ and $PAP^{-1}=\begin{bmatrix}
0 & \cdots & 0 & c_N\\
c_1 & \cdots & 0 & 0\\
\vdots & \ddots & \vdots &\vdots\\
0 & \cdots &c_{N-1} & 0
\end{bmatrix}$, where $c_i$ are positive divisors of $d$ for $1\leq i\leq N$ and $c_1c_2\cdot \cdot \cdot c_{N}=d$. 

Let
$$v:=\prod\limits_{i=1}^{N-1}c_{N-i}^{\frac{i}{N}}=c^{\frac{1}{N}}_{N-1}\cdots c^{\frac{N-2}{N}}_{2} c^{\frac{N-1}{N}}_{1},$$
and
\begin{align*}
[B]:=&(v,\frac{1}{c_{1}}d^{\frac{1}{N}}v,\frac{1}{c_1c_{2}}d^{\frac{2}{N}}v,\dots ,\frac{1}{c_1c_2\cdots c_{N-1}}d^{\frac{N-1}{N}}v)\\
=&(c^{\frac{1}{N}}_{N-1}\cdots c^{\frac{N-2}{N}}_{2} c^{\frac{N-1}{N}}_{1},c^{\frac{1}{N}}_{N}\cdots c^{\frac{N-2}{N}}_{3} c^{\frac{N-1}{N}}_{2},\dots,c^{\frac{1}{N}}_{N-2}\cdots c^{\frac{N-2}{N}}_{1} c^{\frac{N-1}{N}}_{N}),
\end{align*}
Notice that 
\begin{align*}
d^{\frac{1}{N}}[B]=
[B]
\begin{bmatrix}
0 & \cdots & 0 & c_N\\
c_1 & \cdots & 0 & 0\\
\vdots & \ddots & \vdots &\vdots\\
0 & \cdots &c_{N-1} & 0
\end{bmatrix}.
\end{align*}
Hence, by definition,
\begin{align*}
M \cong \mathbb{Z}\langle 
c^{\frac{1}{N}}_{N-1}\cdots c^{\frac{N-2}{N}}_{2} c^{\frac{N-1}{N}}_{1},c^{\frac{1}{N}}_{N}\cdots c^{\frac{N-2}{N}}_{3} c^{\frac{N-1}{N}}_{2},\dots,c^{\frac{1}{N}}_{N-2}\cdots c^{\frac{N-2}{N}}_{1} c^{\frac{N-1}{N}}_{N}
\rangle.
\end{align*}

If $m\geq N$,
when there are exactly $j$ elements equal to $1$ with $0\leq j\leq N-1$, there are $\frac{(N-1)!}{j!}=\frac{(N)_{N-j}}{N}$ different $\mathbb{Z}_{+}\text-$modules. The number of ways to divide the $m$ prime numbers into $N-j$ pairwise disjoint non empty sets is
$S(m,N-j)$ follows from Lemma \ref{simpleLemm5}. Let $k=N-j$, then $1\leq k\leq N$. In this case, the number of irreducible $\mathbb{Z}_{+}\text-$modules is $S(m,k)\frac{(N-1)!}{j!}=S(m,k)\frac{(N)_k}{N}$.
\par
Notice that $(N)_{N-1}=(	N)_{N}$ and $(N)_l=0$, when $l\textgreater N$. Thus, the total number of irreducible $\mathbb{Z}_{+}\text-$modules is $\sum_{k=1}^{N}S(m,k)\frac{(N)_k}{N}=\sum_{k=1}^mS(m,k)\frac{(N)_k}{N}=\frac{N^m}{N}=N^{m-1}$. 
\par
	If $m\textless N$,
when there are exactly $j$ elements equal to $1$, then $N-m\leq j\leq N-1$, there are $\frac{(N-1)!}{j!}=\frac{(N)_{N-j}}{N}$ different $\mathbb{Z}_{+}\text-$module. The number of ways to divide the $m$ prime numbers into $N-j$ pairwise disjoint non empty sets is
$S(m,N-j)$ follows from Lemma \ref{simpleLemm5}. Let $k=N-j$, then $1\leq k\leq m$. In this case, the number of irreducible $\mathbb{Z}_{+}\text-$modules is $S(m,N-j)\frac{(N)_{N-j}}{N}=S(m,k)\frac{(N)_k}{N}$.
\par
Thus, the total number of irreducible $\mathbb{Z}_{+}\text-$modules is $\sum\limits_{k=1}^mS(m,k)\frac{(N)_k}{N}=\frac{N^m}{N}=N^{m-1}$. 
\end{proof}

\begin{exam}
 Let $N=4$ and $d=6$. As shown in Example \ref {ex1}, the decomposition $6=1\times1\times2\times3$ correspond three non-isomorphic irreducible \(\mathbb{Z}_+\)-modules.
 And, the decomposition $6=1\times1\times1\times6$ correspond exactly one irreducible \(\mathbb{Z}_+\)-module:
$
\mathbb{Z}\langle 
1,6^{\frac{1}{4}},6^{\frac{2}{4}},6^{\frac{3}{4}}
\rangle.
$
So $\mathbb{Z}[6^{\frac{1}{4}}]$ has exactly four irreducible \(\mathbb{Z}_+\)-modules up to isomorphism.
\end{exam}

The following lemma is highly analogous to Schur's lemma and plays a crucial role in our subsequent proof.

\begin{lemm}\label{simpleLemm6} Let $d_1,d_2$ be coprime square-free positive integers and $A\in M_N(\mathbb{Z}_{+})$. Assume the matrices $B_1$ and $B_2$ are defined as follows
\begin{align*}
B_1=
\begin{bmatrix}
0 & \cdots & 0 & c_N\\
c_1 & \cdots & 0 & 0\\
\vdots & \ddots & \vdots &\vdots\\
0 & \cdots &c_{N-1} & 0
\end{bmatrix}
,
B_2=
\begin{bmatrix}
0 & \cdots & 0 & c^{\prime}_N\\
c^{\prime}_1 & \cdots & 0 & 0\\
\vdots & \ddots & \vdots &\vdots\\
0 & \cdots &c^{\prime}_{N-1} & 0
\end{bmatrix},
\end{align*}
where $c_1c_2\cdot \cdot \cdot c_{N}=d_1,c^{\prime}_1c^{\prime}_2\cdot \cdot \cdot c^{\prime}_{N}=d_1$ and $c_{N}=max\{c_1,\dots,c_N\}, c^{\prime}_{N}=max\{c^{\prime}_1,\dots,c^{\prime}_N\}$. 
If $B_1 A=A B_2$ and the non-zero elements of $A$ are divisors of $d_2$, then $A=0$ if $B_1 \neq B_2$, and $A$ is a scalar matrix if $B_1=B_2$.
           \begin{proof}Let \( A = (a_{ij})_{1 \leq i,j \leq N} \). If \( B_1 = B_2 \), then \( B_1 A = A B_2 \) implies the system of equations
\[
c_1 a_{11} = c_1 a_{22}, \quad c_2 a_{22} = c_2 a_{33}, \quad \dots, \quad c_N a_{NN} = c_N a_{11},
\]
which imply that all diagonal elements \( a_{ii} \)  are equal for  \( 1 \leq i \leq N \). Let \( k = a_{ii} \)  for all \( 1 \leq i \leq N \) . We then obtain the matrix equation
\[
\begin{bmatrix}
c_N a_{N1} & \cdots & c_N a_{N(N-1)} & c_N k \\
c_1 k & \cdots & c_1 a_{1(N-1)} & c_1 a_{1N} \\
\vdots & \ddots & \vdots & \vdots \\
c_{N-1} a_{(N-1)1} & \cdots & c_{N-1} k & c_{N-1} a_{(N-1)N}
\end{bmatrix}
=
\begin{bmatrix}
c_1 a_{12} & \cdots & c_{N-1} a_{1N} & c_N k \\
c_1 k & \cdots & c_{N-1} a_{2N} & c_N a_{21} \\
\vdots & \ddots & \vdots & \vdots \\
c_1 a_{N2} & \cdots & c_{N-1} k & c_N a_{N1}
\end{bmatrix}.
\]
By comparing the last column of the two matrices, we have
\[
c_1 a_{1N} = c_N a_{21}, \quad c_2 a_{2N} = c_N a_{31}, \quad \dots, \quad c_{N-1} a_{(N-1)N} = c_N a_{N1}.
\]
As \( c_N \neq c_i \) for \( 1 \leq i \leq N-1 \), and $d_1,d_2$ are coprime to each other, it follows that \( a_{i1} = a_{jN} = 0 \) for \( i \neq 1 \) and \( j \neq N \). Consequently, \( A \) takes the form
\[
\begin{bmatrix}
k & a_{12} & \cdots & a_{1(N-1)} & 0 \\
0 & k & \cdots & a_{2(N-1)} & 0 \\
\vdots & \vdots & \ddots & \vdots & \vdots \\
0 & a_{(N-1)2} & \cdots & k & 0 \\
0 & a_{N2} & \cdots & a_{N(N-1)} & k
\end{bmatrix}.
\]
The equation \( B_1 A = A B_1 \) then becomes
\[
\begin{bmatrix}
0 & \cdots & c_N a_{N(N-1)} & c_N k \\
c_1 k & \cdots & c_1 a_{1(N-1)} & 0 \\
\vdots & \vdots & \vdots & \vdots \\
0 & \cdots & c_{N-2} a_{(N-2)(N-1)} & 0 \\
0 & \cdots & c_{N-1} k & 0
\end{bmatrix}
=
\begin{bmatrix}
c_1 a_{12} & \cdots & 0 & c_N k \\
c_1 k & \cdots & 0 & 0 \\
\vdots & \vdots & \vdots & \vdots \\
c_1 a_{(N-2)2} & \cdots & 0 & 0 \\
c_1 a_{N2} & \cdots & c_{N-1} k & 0
\end{bmatrix}.
\]
By comparing the first and second-to-last columns of the two matrices, we find that \( a_{i2} = a_{j(N-1)} = 0 \) for \( i \neq 2 \) and \( j \neq N-1 \). Thus, \( A \) can be simplified to
\[
\begin{bmatrix}
k & 0 & a_{13} & \cdots & a_{1(N-2)} & 0 & 0 \\
0 & k & a_{23} & \cdots & a_{2(N-2)} & 0 & 0 \\
0 & 0 & k & \cdots & a_{3(N-2)} & 0 & 0 \\
\vdots & \vdots & \vdots & \ddots & \vdots & \vdots \\
0 & 0 & a_{(N-2)3} & \cdots & k & 0 & 0 \\
0 & 0 & a_{(N-1)3} & \cdots & a_{(N-1)(N-2)} & k & 0 \\
0 & 0 & a_{N3} & \cdots & a_{N(N-2)} & 0 & k
\end{bmatrix}.
\]
Continuing this process, we know that \( a_{ij} = 0 \) for \( i \neq j \). Therefore, \( A \) is a scalar matrix.

If \( B_1 \neq B_2 \), it follows that there exist coefficients \( c_i \) and \( c'_i \) such that \( c_i \neq c'_i \) for  some \( 1 \leq i \leq N \). The equations \( c_i a_{ii} = c'_i a_{(i+1)(i+1)} \) show that \( a_{ii} = a_{(i+1)(i+1)} = 0 \), so implying that all diagonal elements of matrix \( A \) are zero. If \( c'_N \geq c_N \), the method previously outlined is applicable. If \( c'_N < c_N \), the same method can be employed, with the modification of interchanging columns and rows. Consequently, it follows that \( a_{ij} = 0 \) for all \( 1 \leq i, j \leq N \).
           \end{proof}

\end{lemm}

\begin{theo}\label{TypeRank}
Assume that $d_1,d_2,\dots,d_r$ are pairwise relatively prime, positive, square-free integers and $d_j\geq2$ for all $1\leq j \leq r$. Let $M$ be an irreducible $\mathbb{Z}_{+}\text-$module of $R=\mathbb{Z}[d^{\frac{1}{N_1}}_1,d^{\frac{1}{N_2}}_2,\dots,d^{\frac{1}{N_r}}_r]$, where $N_i\geq2$ for all $1\leq i \leq r$. Then\begin{enumerate}
   \item rank$(M)=N_1N_2\cdots N_r$.
   \item for any $1\leq j \leq r$, let $\widehat{N_j}:=\frac{N_1N_2\cdots N_r}{N_j}$ , then $M\cong \bigoplus  ^{\widehat{N_j}}_{i=1}M_i$ as $\mathbb{Z}_{+}\text-$module of  $R=\mathbb{Z}[d^{\frac{1}{N_j}}_j]$. Moreover, there exist positive divisors $c_{1j},c_{2j},\dots,c_{N_jj}$ such that $c_{1j}c_{2j}\cdot \cdot \cdot c_{N_jj}=d_j$, and for all $1\leq i \leq \widehat{N_j}$,
\par
 $M_i\cong\mathbb{Z}\langle 
c^{\frac{1}{N_j}}_{(N_j-1)j}\cdots c^{\frac{N_j-2}{N_j}}_{2j} c^{\frac{N_j-1}{N_j}}_{1j},c^{\frac{1}{N_j}}_{N_jj}\cdots c^{\frac{N_j-2}{N_j}}_{3j} c^{\frac{N_j-1}{N_j}}_{2j},\dots,c^{\frac{1}{N_j}}_{(N_j-2)j}\cdots c^{\frac{N_j-2}{N_j}}_{1j} c^{\frac{N_j-1}{N_j}}_{N_jj}
\rangle$.
    \end{enumerate}

\end{theo}
\begin{proof}
For any \( 1 \leq j \leq r \), by Corollary \ref{resultonZ[d]}, we can find a \(\mathbb{Z}_+\)-basis for \( M \) such that the non-negative integer matrix determined by \( d_j^{\frac{1}{N_j}} \) is given by $
B = \begin{bmatrix}
B_1 &  &  & \\
 & B_2 &  &\\
 &  & \ddots &\\
 &  &  & B_g
\end{bmatrix},
$
where 
$
B_k = \begin{bmatrix}
0 & \cdots & 0 & c_{N_jk}\\
c_{1k} & \cdots & 0 & 0\\
\vdots & \ddots & \vdots &\vdots\\
0 & \cdots &c_{(N_j-1)k} & 0
\end{bmatrix},
$
\( c_{1k}\cdots c_{N_jk} = d_j, c_{N_jk} = \max\{c_{1k}, \dots, c_{N_jk}\} \) for all \( 1 \leq k \leq g \). Moreover, without loss of generality, we further  assume that $B_1,\dots,B_g=\underbrace{B_1,\dots,B_1}_{l_1},
\underbrace{B_2,\dots,B_2}_{l_2},\dots,\underbrace{B_s,\dots,B_s}_{l_s}$,          and $l_1+l_2+\dots +l_s=g$.  
 Assume the matrix determined by $d^{\frac{1}{N_l}}_l$ is $A=\begin{bmatrix}
A_{11} & A_{12} & \cdots & A_{1g}\\
A_{21} & A_{22} & \cdots & A_{2g}\\
\vdots & \vdots & \ddots &\vdots\\
A_{g1} & A_{g2} &\cdots & A_{gg}
\end{bmatrix}$, 
where $1\leq l\neq j \leq r$. Because $R$ is commutative, $d^{\frac{1}{N_j}}_j d^{\frac{1}{N_l}}_l=d^{\frac{1}{N_l}}_ld^{\frac{1}{N_j}}_j$ implies $AB=BA$. Hence, we deduce from Lemma \ref{simpleLemm6} that there exists matrices $A_1,\dots,A_s$ such that $A=\begin{bmatrix}
A_{1} & 0 & \cdots & 0\\
0 & A_{2} & \cdots & 0\\
\vdots & \vdots & \ddots &\vdots\\
0 & 0 &\cdots & A_{s}
\end{bmatrix}
$.

Notice that the matrices determined by $d^{\frac{1}{N_k}}_k(1\leq k\neq j \leq r)$ all have the same shape as $A$, hence $s=1$, otherwise $M$ is a direct sum of non-trivial irreducible $\mathbb{Z}_{+}\text-$modules of $R$, which contradicts the fact that $M$ is irreducible. That is to say, as a $\mathbb{Z}_{+}\text-$module of $\mathbb{Z}[d_j^{\frac{1}{N_j}}]$, $M$ is isomorphic to a direct sum of 
\begin{align*}
\mathbb{Z}\langle 
c^{\frac{1}{N_j}}_{(N_j-1)j}\cdots c^{\frac{N_j-2}{N_j}}_{2j} c^{\frac{N_j-1}{N_j}}_{1j},c^{\frac{1}{N_j}}_{N_jj}\cdots c^{\frac{N_j-2}{N_j}}_{3j} c^{\frac{N_j-1}{N_j}}_{2j},\dots,c^{\frac{1}{N_j}}_{(N_j-2)j}\cdots c^{\frac{N_j-2}{N_j}}_{1j} c^{\frac{N_j-1}{N_j}}_{N_jj}
\rangle,
\end{align*}
 for some positive divisors $c_{1j},c_{2j},\dots,c_{N_jj}$ of $d_j$.
\par 

Assume that $\{m_1,\dots,m_n\}$ is a $\mathbb{Z}_{+}\text-$basis of $M$, it follows from the above conclusion that there exist positive integers $c_{N_11}= \mathrm{max}\{c_{11},\dots,c_{(N_1-1)1},c_{N_11}\}$ such that  the corresponding matrix determined by $d_1^{\frac{1}{N_1}}$ has the form
$D_1=\begin{bmatrix}
C_1 &  &  & \\
 & C_1 &  &\\
 &  & \ddots &\\
 &  &  & C_1
\end{bmatrix},$
where $$C_1=
\begin{bmatrix}
0 & \cdots & 0 & c_{N_11}\\
c_{11} & \cdots & 0 & 0\\
\vdots & \ddots & \vdots &\vdots\\
0 & \cdots &c_{(N_1-1)1} & 0
\end{bmatrix}.$$
Note that for $2 \leq i \leq r$, $d_i$ is coprime to $d_1$. Lemma \ref{simpleLemm6} and Lemma \ref{simpleLemm2} state the matrix determined by $d_i^{\frac{1}{N_i}}$ has the following form $$A_i=\begin{bmatrix}
0 & A_{12} & \cdots & A_{1t}\\
A_{21} & 0 & \cdots & A_{2t}\\
\vdots & \vdots & \ddots &\vdots\\
A_{t1} & A_{t2} &\cdots & 0
\end{bmatrix},$$ where $A_{kl}$ are scalar matrices and $t=\frac{n}{N_1}$. Moreover, for any $1\leq k\leq t$, Corollary \ref{resultonZ[d]}  means there exists exactly one $1\leq l\leq t$ such that $A_{kl}\neq 0$. These $t$ non-zero scalars are denoted by $k_{1i},k_{2i},\dots,k_{ti}$, which are allowed to be same. For each $i$, we can find a permutation matrix $P_i$ such that $D_i=P_iA_iP_i^{-1}=\begin{bmatrix}
C_i &  &  & \\
 & C_i &  &\\
 &  & \ddots &\\
 &  &  & C_i
\end{bmatrix}$, $C_i=\begin{bmatrix}
0 & \cdots & 0 & c_{N_ii}\\
c_{1i} & \cdots & 0 & 0\\
\vdots & \ddots & \vdots &\vdots\\
0 & \cdots &c_{(N_i-1)i} & 0
\end{bmatrix}$ and $C_i$ repeats $\frac{n}{N_i}$ times. Therefore,
 $\{c_{1i},c_{2i},\dots,c_{N_{i}i}\}\subseteq\{k_{1i},k_{2i},\dots,k_{ti}\}$. For any $i$ with $2\leq i\leq r$, $D_1A_i$ has the same form as $A_i$, this means that the non-zero blocks of both in the same position, so 
\begin{align*}
\bigcup\limits_{p=1}^{N_i}\{c_{pi}c_{11},c_{pi}c_{21},\dots,c_{pi}c_{N_11}\}\subseteq \bigcup\limits_{q=1}^{t}\{k_{qi}c_{11},k_{qi}c_{21},\dots,k_{qi}c_{N_11}\}.
\end{align*} 
And $D_1A_i$ has exactly one non-zero element in each row and column, so all $c_{j_{i}i}c_{j_{1}1}$ are coefficients of
$D_1A_i$, where $1\leq j_1 \leq N_1$ and $1\leq j_i \leq N_i$.
Notice $d_1^{\frac{1}{N_1}}d_2^{\frac{1}{N_2}}\cdots d_r^{\frac{1}{N_r}}$ corresponds to the matrix of the form $D_1A_2A_3\cdots A_r$, by the commutative and associative law, we derive $$(D_1A_2)A_3\cdots A_r=(D_1A_3)A_2\cdots A_r=\cdots=(D_1A_r)A_2\cdots A_{r-1}.$$ Observe that the product $D_1A_2A_3\cdots A_r$ has exactly one non-zero element in each row and column, so 
all $c_{i_{1}1}c_{i_{2}2}\cdots c_{i_{r}r}$ are coefficients of
$D_1A_2A_3\cdots A_r$, where $1\leq i_l \leq N_l$ for all $1\leq l\leq r$. This implies that rank$(M)=n\geq N_1N_2\cdots N_r$.
\par
It follows from Corollary \ref{resultonZ[d]} that there exist a unique non-zero element in each row and column of the non-negative matrix determined by $d^{\frac{i_1}{N_{j_1}}}_{j_1}d^{\frac{i_2}{N_{j_2}}}_{j_2}\cdots d^{\frac{i_{s}}{N_{j_{s}}}}_{j_{s}}$, where $0\leq i_s\leq N_{j_{s}}-1$ for all $1\leq s \leq r$. Then it is easy to see that the rank of the generated $\mathbb{Z}_{+}\text-$submodule $m_1$ of $M$ is at most $N_1N_2\cdots N_r$. Since $M$ is irreducible, $M=(m_1)$ and rank$(M)=N_1N_2\cdots N_r$. 

Consequently, as a $\mathbb{Z}_{+}\text-$module of $\mathbb{Z}[d_j^{\frac{1}{N_j}}]$, we have
$
M\cong\bigoplus ^{\widehat{N_j}}_{i=1}M_i$,
and 
\begin{align*}
M_i\cong\mathbb{Z}\langle 
c^{\frac{1}{N_j}}_{(N_j-1)j}\cdots c^{\frac{N_j-2}{N_j}}_{2j} c^{\frac{N_j-1}{N_j}}_{1j},c^{\frac{1}{N_j}}_{N_jj}\cdots c^{\frac{N_j-2}{N_j}}_{3j} c^{\frac{N_j-1}{N_j}}_{2j},\dots,c^{\frac{1}{N_j}}_{(N_j-2)j}\cdots c^{\frac{N_j-2}{N_j}}_{1j} c^{\frac{N_j-1}{N_j}}_{N_jj}
\rangle,
\end{align*}
for all $1\leq i \leq \widehat{N_j}$, where $\widehat{N_j}=\frac{N_1N_2\cdots N_r}{N_j}.
$
This completes the proof of the theorem.
\end{proof}

\begin{rema}\label{Example}A part of the conclusion of Theorem \ref{TypeRank} becomes invalid if \(d_i\) and \(d_j\) share a non-trivial common divisor. For instance, consider the ring \(R = \mathbb{Z}[6^{\frac{1}{3}}, 10^{\frac{1}{3}}]\). It is evident that \(R\) is irreducible as a \(\mathbb{Z}_+\)-module of itself. As a \(\mathbb{Z}_+\)-module of \(\mathbb{Z}[6^{\frac{1}{3}}]\), we have:
\[
R \cong \mathbb{Z}\langle 1, 6^{\frac{1}{3}}, 6^{\frac{2}{3}} \rangle \oplus \mathbb{Z}\langle 2^{\frac{1}{3}}, 3^{\frac{1}{3}}2^{\frac{2}{3}}, 3^{\frac{2}{3}} \rangle \oplus \mathbb{Z}\langle 2^{\frac{2}{3}}, 3^{\frac{1}{3}}, 2^{\frac{1}{3}}3^{\frac{2}{3}} \rangle,
\]
yet, \(\mathbb{Z}\langle 1, 6^{\frac{1}{3}}, 6^{\frac{2}{3}} \rangle \not\cong \mathbb{Z}\langle 2^{\frac{1}{3}}, 3^{\frac{1}{3}}2^{\frac{2}{3}}, 3^{\frac{2}{3}} \rangle \not\cong \mathbb{Z}\langle 2^{\frac{2}{3}}, 3^{\frac{1}{3}}, 2^{\frac{1}{3}}3^{\frac{2}{3}} \rangle\) as \(\mathbb{Z}_+\)-modules of \(\mathbb{Z}[6^{\frac{1}{3}}]\).
\end{rema}

The following theorem generalizes the conclusion of  Corollary \ref{resultonZ[d]}.

\begin{theo}\label{maintheorem}
Let $R=\Z[d^{\frac{1}{N_1}}_1,d^{\frac{1}{N_2}}_2,\dots,d^{\frac{1}{N_r}}_r]$, where $d_1,d_2,\dots,d_r$ are pairwise relatively prime positive square-free integers with $d_j\geq 2$ for all $1\leq j \leq r$, and $N_i\geq2$ for all $1\leq i \leq r$. Then up to isomorphism $R$ has exactly $N_1^{m_1-1}N_2^{m_2-1}\cdots N_r^{m_r-1}$ irreducible $\mathbb{Z}_{+}\text-$modules,  where $m_j$ is the number of prime divisors of $d_j$. 
\end{theo}
\begin{proof}
  Let \(d_j = c_{1j} c_{2j} \cdots c_{N_jj}\), \(c_{N_jj} = \max\{c_{1j}, \dots, c_{N_jj}\}\) for \(1 \leq j \leq r\). Let
\begin{align*}
v_1:=&(\prod\limits_{i=1}^{N_1-1}c_{(N_1-i)1}^{\frac{i}{N_1}})(\prod\limits_{i=1}^{N_2-1}c_{(N_2-i)2}^{\frac{i}{N_2}})\cdots(\prod\limits_{i=1}^{N_r-1}c_{(N_r-i)r}^{\frac{i}{N_r}})\\
=&(c^{\frac{1}{N_1}}_{(N_1-1)1}\cdots c^{\frac{N_1-2}{N_1}}_{21}c^{\frac{N_1-1}{N_1}}_{11})(c^{\frac{1}{N_2}}_{(N_2-1)2}\cdots c^{\frac{N_2-2}{N_2}}_{22}c^{\frac{N_2-1}{N_2}}_{12})\cdots (c^{\frac{1}{N_r}}_{(N_r-1)r}\cdots c^{\frac{N_r-2}{N_r}}_{2r}c^{\frac{N_r-1}{N_r}}_{1r}),
\end{align*}
then $v_1$ generates a $\mathbb{Z}_{+}\text-$module $T(c_{11},\dots,c_{(N_1-1)1}):=(v_1)$ of $\mathbb{Z}[d_1^{\frac{1}{N_1}}]$, moreover it
is irreducible with 
$\mathbb{Z}_{+}\text-$basis
\begin{align*}
\{v_1,v_2=\frac{1}{c_{11}}d_1^{\frac{1}{N_1}}v_1,v_3=\frac{1}{c_{11}c_{21}}d_1^{\frac{2}{N_1}}v_1,\dots ,v_{N_1}=\frac{1}{c_{11}c_{21}\cdots c_{(N_1-1)1}}d_1^{\frac{N_1-1}{N_1}}v_1\}.
\end{align*}
 Then it is easy to see that 
$$\{v_1,v_2,\dots,v_{N_1},\frac{1}{c_{12}}d_2^{\frac{1}{N_2}}v_1,\frac{1}{c_{12}}d_2^{\frac{1}{N_2}}v_2,\dots,\frac{1}{c_{12}}d_2^{\frac{1}{N_2}}v_{N_1},\\$$
$$\frac{1}{c_{12}c_{22}}d_2^{\frac{2}{N_2}}v_1,\frac{1}{c_{12}c_{22}}d_2^{\frac{2}{N_2}}v_2,\dots,\frac{1}{c_{12}c_{22}}d_2^{\frac{2}{N_2}}v_{N_1},\dots,\\$$
$$\frac{1}{c_{12}c_{22}\cdots c_{(N_2-1)2}}d_2^{\frac{N_2-1}{N_2}}v_1,\frac{1}{c_{12}c_{22}\cdots c_{(N_2-1)2}}d_2^{\frac{N_2-1}{N_2}}v_2,\dots,\frac{1}{c_{12}c_{22}\cdots c_{(N_2-1)2}}d_2^{\frac{N_2-1}{N_2}}v_{N_1}\}$$
generate an irreducible $\mathbb{Z}_{+}\text-$module $$T((c_{11},\dots,c_{(N_1-1)1}),(c_{12},\dots,c_{(N_2-1)2}))$$ of 
$\mathbb{Z}[d_1^{\frac{1}{N_1}},d_2^{\frac{1}{N_2}}]$. By doing this construction inductively, we get a $\mathbb{Z}_{+}\text-$module of the domain $R$, which is determined by the sequence
$
((c_{11},\dots,c_{(N_1-1)1}),\dots,(c_{1r},\dots,c_{(N_r-1)r})),
$
and will be denoted by 
$$T((c_{11},\dots,c_{(N_1-1)1}),\dots,(c_{1r},\dots,c_{(N_r-1)r}))$$
below. Notice that
\begin{align*}
\mathrm{rank}(T((c_{11},\dots,c_{(N_1-1)1}),\dots,(c_{1r},\dots,c_{(N_r-1)r})))=N_1N_2\cdots N_r
\end{align*}
 by construction, then $T((c_{11},\dots,c_{(N_1-1)1}),\dots,(c_{1r},\dots,c_{(N_r-1)r}))$ must be an irreducible $\mathbb{Z}_{+}\text-$module by Theorem \ref{TypeRank}.
\par
Let $\{m_1,m_2,\dots,m_n\}$ be a $\mathbb{Z}_{+}\text-$basis of $M$, then for any $1\leq j\leq r$, it follows from Theorem \ref{TypeRank} that $n=N_1N_2\cdots N_r$ and 
$M\cong \bigoplus ^{\widehat{N_j}}_{i=1}M_i$, where
\begin{align*}
M_i\cong\mathbb{Z}\langle 
c^{\frac{1}{N_j}}_{(N_j-1)j}\cdots c^{\frac{N_j-2}{N_j}}_{2j} c^{\frac{N_j-1}{N_j}}_{1j},c^{\frac{1}{N_j}}_{N_jj}\cdots c^{\frac{N_j-2}{N_j}}_{3j} c^{\frac{N_j-1}{N_j}}_{2j},\dots,c^{\frac{1}{N_j}}_{(N_j-2)j}\cdots c^{\frac{N_j-2}{N_j}}_{1j} c^{\frac{N_j-1}{N_j}}_{N_jj}
\rangle
\end{align*}
 as $\mathbb{Z}[d^{\frac{1}{N_j}}_j]\text-$module, and $d_j=c_{1j}c_{2j}\cdots c_{N_jj}$,  \(c_{N_jj} = \max\{c_{1j}, \dots, c_{N_jj}\}\) for $1\leq j \leq r$. Thus, $M$ determines a sequence of positive integers 
\begin{align*}
((c_{11},\dots,c_{(N_1-1)1}),\dots,(c_{1r},\dots,c_{(N_r-1)r})).
\end{align*}
Then we deduce from the above construction that the generated submodule $(m_1)$ of $M$ is isomorphic to
$T((c_{11},\dots,c_{(N_1-1)1}),\dots,(c_{1r},\dots,c_{(N_r-1)r}))$. As $M$ is irreducible
\begin{align*}
M=T((c_{11},\dots,c_{(N_1-1)1}),\dots,(c_{1r},\dots,c_{(N_r-1)r})).
\end{align*}
Up to isomorphism, for any $1\leq j \leq r$, it follows from Corollary \ref{resultonZ[d]} that there are exactly $N_j^{m_j-1}$ irreducible $\mathbb{Z}_{+}\text-$modules over $\mathbb{Z}[d_j^{\frac{1}{N_j}}]$, where $m_j$ is the number of prime divisors of $d_j$. Moreover, by the definition of the 
$\mathbb{Z}_{+}\text-$module 
$
T((c_{11},\dots,c_{(N_1-1)1}),\dots,(c_{1r},\dots,c_{(N_r-1)r})),
$
it is easy to show that $$T((c_{11},\dots,c_{(N_1-1)1}),\dots,(c_{1r},\dots,c_{(N_r-1)r}))\cong T((b_{11}, \dots ,b_{(N_1-1)1}),\dots,(b_{1r}, \dots,b_{(N_r-1)r}))$$ as $\mathbb{Z}_{+}\text-$module of $R$ if and only if $c_{ij}=b_{ij}$, where $1\leq i\leq N_j-1$ for all $1\leq j\leq r$, hence the number of isomorphism classes of irreducible $\mathbb{Z}_{+}\text-$modules $$T((c_{11},\dots,c_{(N_1-1)1}),\dots,(c_{1r},\dots,c_{(N_r-1)r}))$$ is $N_1^{m_1-1}N_2^{m_2-1}\cdots N_r^{m_r-1}$. This completes the proof of the theorem.
\end{proof}

\begin{rema}\label{Example2} The use of tensor products facilitates a more convenient representation. By the definition of $\otimes_{\Z}$, we have the free $\Z$-module isomorphism
\begin{align*}
M_1\otimes_{\Z}M_2\otimes_{\Z}\cdots \otimes_{\Z}M_r &  \xrightarrow{\thicksim} T((c_{11},\dots,c_{(N_1-1)1}),\dots,(c_{1r},\dots,c_{(N_r-1)r}))=M\\
m_1\otimes m_2\otimes \cdots \otimes m_r &\mapsto m_1m_2\cdots m_r
,
\end{align*}
where $m_j\in M_j$ for $1\leq j\leq r,$
$$
M_j\cong\mathbb{Z}\langle 
c^{\frac{1}{N_j}}_{(N_j-1)j}\cdots c^{\frac{N_j-2}{N_j}}_{2j} c^{\frac{N_j-1}{N_j}}_{1j},c^{\frac{1}{N_j}}_{N_jj}\cdots c^{\frac{N_j-2}{N_j}}_{3j} c^{\frac{N_j-1}{N_j}}_{2j},\dots,c^{\frac{1}{N_j}}_{(N_j-2)j}\cdots c^{\frac{N_j-2}{N_j}}_{1j} c^{\frac{N_j-1}{N_j}}_{N_jj}
\rangle.
$$

Using the $\Z$-algebra isomorphism
\begin{align*}
\Z[d^{\frac{1}{N_1}}_1]\otimes_{\Z}\Z[d^{\frac{1}{N_2}}_2]\otimes_{\Z}\cdots \otimes_{\Z}\Z[d^{\frac{1}{N_r}}_r] &  \xrightarrow{\thicksim} \Z[d^{\frac{1}{N_1}}_1,d^{\frac{1}{N_2}}_2,\dots,d^{\frac{1}{N_r}}_r]=R\\
s_1\otimes s_2\otimes \cdots \otimes s_r &\mapsto s_1s_2\cdots s_r.
\end{align*}
The action of $R\cong  \Z[d^{\frac{1}{N_1}}_1]\otimes_{\Z}\Z[d^{\frac{1}{N_2}}_2]\otimes_{\Z}\cdots \otimes_{\Z}\Z[d^{\frac{1}{N_r}}_r]$ on $M$ is determined by
\begin{align*}
\Z[d^{\frac{1}{N_1}}_1]\otimes_{\Z}\cdots \otimes_{\Z}\Z[d^{\frac{1}{N_r}}_r]\times M_1\otimes_{\Z}M_2\otimes_{\Z}\cdots \otimes_{\Z}M_r &\xrightarrow{\thicksim} M_1\otimes_{\Z}M_2\otimes_{\Z}\cdots \otimes_{\Z}M_r \\
s_1\otimes s_2\otimes \cdots \otimes s_r \times m_1\otimes m_2\otimes \cdots \otimes m_r &\mapsto s_1m_1\otimes s_2m_2\otimes \cdots \otimes s_rm_r.
\end{align*}
Under the action of $R$, the above $\Z$-module isomorphism lifts to $\Z_{+}$-module isomorphism.

Thus, we can regard \( M \) as an irreducible \(\mathbb{Z}_+\)-module of $$R\cong  \Z[d^{\frac{1}{N_1}}_1]\otimes_{\Z}\Z[d^{\frac{1}{N_2}}_2]\otimes_{\Z}\cdots \otimes_{\Z}\Z[d^{\frac{1}{N_r}}_r].$$ Then $M$ admits a canonical tensor product decomposition
$$M\cong M_1\otimes_{\Z}M_2\otimes_{\Z}\cdots \otimes_{\Z}M_r,$$
where $M_i$ is an irreducible $\mathbb{Z}_{+}\text-$module of $\mathbb{Z}[d^{\frac{1}{N_i}}_i]$ for $1\leq i \leq r$.

\end{rema}

\begin{exam}
Let $R=\Z[6^{\frac{1}{2}},35^{\frac{1}{3}}]$, and let $M$ be an irreducible $\Z_+$-module of $R$, then $M\cong T(c_{11},c_{12},c_{22})$ by Theorem \ref{maintheorem}, where $c_{11}$ is a divisor of $6$ and $c_{12},c_{22}$ are divisors of $35$. Hence, $M$ is isomorphic to one of the following  irreducible $\Z_+$-modules:
\begin{align*}
&T(1,1,1)\cong\Z\langle 1,6^{\frac{1}{2}},35^{\frac{1}{3}},6^{\frac{1}{2}}35^{\frac{1}{3}}, 35^{\frac{2}{3}},6^{\frac{1}{2}}35^{\frac{2}{3}}\rangle \cong\Z [6^{\frac{1}{2}}] \otimes _{\Z}\Z[35^{\frac{1}{3}}]\cong R,\\
&T(1,1,5)\cong\Z\langle 5^{\frac{1}{3}},6^{\frac{1}{2}}5^{\frac{1}{3}},7^{\frac{1}{3}}5^{\frac{2}{3}},6^{\frac{1}{2}}7^{\frac{1}{3}}5^{\frac{2}{3}}, 7^{\frac{2}{3}},6^{\frac{1}{2}}7^{\frac{2}{3}}\rangle\cong\Z\langle 1,6^{\frac{1}{2}}\rangle \otimes _{\Z}\Z\langle 5^{\frac{1}{3}},7^{\frac{1}{3}}5^{\frac{2}{3}},7^{\frac{2}{3}}\rangle , \\
&T(1,5,1)\cong\Z\langle 5^{\frac{2}{3}},6^{\frac{1}{2}}5^{\frac{2}{3}},7^{\frac{1}{3}},6^{\frac{1}{2}}7^{\frac{1}{3}},5^{\frac{1}{3}}7^{\frac{2}{3}},6^{\frac{1}{2}}5^{\frac{1}{3}}7^{\frac{2}{3}}\rangle \cong\Z\langle 1,6^{\frac{1}{2}}\rangle \otimes _{\Z}\Z\langle 5^{\frac{2}{3}},7^{\frac{1}{3}},5^{\frac{1}{3}}7^{\frac{2}{3}}\rangle , \\
&T(2,1,1)\cong\Z\langle 2^{\frac{1}{2}},3^{\frac{1}{2}},2^{\frac{1}{2}}35^{\frac{1}{3}},3^{\frac{1}{2}}35^{\frac{1}{3}}, 2^{\frac{1}{2}}35^{\frac{2}{3}},3^{\frac{1}{2}}35^{\frac{2}{3}}\rangle\cong\Z\langle 2^{\frac{1}{2}},3^{\frac{1}{2}}\rangle \otimes _{\Z}\Z\langle1,35^{\frac{1}{3}},35^{\frac{2}{3}}\rangle ,\\
&T(2,1,5)\cong\Z\langle 2^{\frac{1}{2}}5^{\frac{1}{3}},3^{\frac{1}{2}}5^{\frac{1}{3}},2^{\frac{1}{2}}7^{\frac{1}{3}}5^{\frac{2}{3}},3^{\frac{1}{2}}7^{\frac{1}{3}}5^{\frac{2}{3}}, 2^{\frac{1}{2}}7^{\frac{2}{3}},3^{\frac{1}{2}}7^{\frac{2}{3}}\rangle\cong\Z\langle 2^{\frac{1}{2}},3^{\frac{1}{2}}\rangle \otimes _{\Z}\Z\langle5^{\frac{1}{3}},7^{\frac{1}{3}}5^{\frac{2}{3}},7^{\frac{2}{3}}\rangle ,\\
&T(2,5,1)\cong\Z\langle 2^{\frac{1}{2}}5^{\frac{2}{3}},3^{\frac{1}{2}}5^{\frac{2}{3}}, 2^{\frac{1}{2}}7^{\frac{1}{3}},3^{\frac{1}{2}}7^{\frac{1}{3}},2^{\frac{1}{2}}5^{\frac{1}{3}}7^{\frac{2}{3}},3^{\frac{1}{2}}5^{\frac{1}{3}}7^{\frac{2}{3}}\rangle\cong\Z\langle 2^{\frac{1}{2}},3^{\frac{1}{2}}\rangle \otimes _{\Z}\Z\langle 5^{\frac{2}{3}},7^{\frac{1}{3}},5^{\frac{1}{3}}7^{\frac{2}{3}}\rangle.
\end{align*}
\end{exam}

\begin{coro}
 Assume $p_1,p_2,\dots,p_r$ are distinct primes. Let $R=\mathbb{Z}[p_1^{\frac{1}{N_1}},p_2^{\frac{1}{N_2}}, \dots,p_r^{\frac{1}{N_r}}]$, where $N_i\geq2$ for all $1\leq i \leq r$. Let $M$ be an irreducible $\mathbb{Z}_{+}\text-$module of $R$. Then $M\cong R$ as  $\mathbb{Z}_{+}\text-$module.
\end{coro}

\end{document}